\documentclass[11pt]{amsart}

\usepackage{times}
\usepackage{amssymb}
\usepackage{graphicx,xspace}
\usepackage{epsfig}
\usepackage{enumitem}
\usepackage[usenames,dvipsnames]{xcolor}
\usepackage{tikz}
\usepackage{gensymb}
\usepackage[T1]{fontenc}
\usepackage[utf8]{inputenc}
\usepackage{bm}
\usepackage{genyoungtabtikz}
\usepackage[config, labelfont={normalsize}]{caption,subfig}
\usepackage{mathrsfs}

\usepackage{hyperref}
\usepackage{todonotes}

\usetikzlibrary{matrix,arrows,calc}
\usetikzlibrary{decorations.markings}
\usetikzlibrary{patterns}
\usetikzlibrary{snakes}
\usetikzlibrary{shapes}

\usepackage[capitalize]{cleveref}

\theoremstyle{plain}

\newtheorem{lemma}{Lemma}[section]
\newtheorem{theorem}[lemma]{Theorem}
\newtheorem{corollary}[lemma]{Corollary}
\newtheorem{proposition}[lemma]{Proposition}

\newtheorem{conjecture}[lemma]{Conjecture}

\theoremstyle{remark}
\newtheorem*{remark}{Remark}

\newtheorem{definition}[lemma]{Definition}

\newcommand{\N}{\mathbb{N}}

\newcommand{\Z}{\mathbb{Z}}

\newcommand{\QQ}{\mathbb{Q}}

\newcommand{\J}{\mathcal{J}}

\newcommand{\Y}{\mathbb{Y}}

\newcommand{\kumu}{\kappa}

\newcommand{\ev}{\text{ev}}
\newcommand{\odd}{\text{odd}}
\newcommand{\PPP}{\mathcal{P}}

\def\la{\lambda}
\def\ka{\kappa}
\def\a{\alpha}
\def\si{\sigma}

\def\r{r}
\DeclareMathOperator{\RHS}{RHS}
\DeclareMathOperator{\LHS}{LHS}
\DeclareMathOperator{\Der}{Der}

\DeclareMathOperator{\IE}{InEx}

\DeclareMathOperator{\supp}{supp}

\DeclareMathOperator{\Id}{Id}

\DeclareMathOperator{\Symm}{Sym}

\def\uu{\bm{u}}
\def\vv{\bm{v}}

\author[M.~Dołęga]{Maciej Dołęga}
\address{
Wydział Matematyki i Informatyki, 
Uniwersytet im.~Adama Mickiewicza, 
Collegium Mathematicum,
Umultowska 87, 
61-614 Poznań, 
Poland, \newline \indent Instytut Matematyczny,
Uniwersytet Wrocławski,  \mbox{pl.\ Grunwaldzki~2/4,} 50-384
Wrocław, Poland}
\email{maciej.dolega@amu.edu.pl}

 \thanks{
MD is supported from {\it NCN}, grant UMO-2015/16/S/ST1/00420.}



\title[Strong factorization property of Macdonald polynomials]{Strong factorization property of Macdonald polynomials and
higher--order Macdonald's positivity conjecture}

\begin{document}

\maketitle

\begin{abstract}
We prove a strong factorization property of interpolation Macdonald
polynomials when $q$ tends to $1$. As a consequence, we show that
Macdonald polynomials have  a strong factorization property when $q$
tends to $1$, which was posed as an open question in our previous
paper with F\'eray. Furthermore, we introduce multivariate $q,t$-Kostka
numbers and we show that
they are polynomials in $q,t$ with integer coefficients by using
the strong factorization property of Macdonald polynomials. We conjecture that multivariate $q,t$-Kostka
numbers
are in fact polynomials in $q,t$ with nonnegative integer coefficients, which
generalizes the celebrated Macdonald's positivity conjecture.
\end{abstract}

\section{Introduction}
\label{sec:introduction}

\subsection{Macdonald polynomials}
\label{subsec:Macdonald}

In 1988 Macdonald \cite{Macdonald1988, Macdonald1995} introduced a new
family of symmetric functions $J_\la^{(q,t)}(\bm{x})$ depending upon
a partition $\la$, a set of variables $\bm{x} = \{x_1,\dots,x_N\}$ and two real parameters $q, t$. They were immediately
hailed as a breakthrough in symmetric function theory as well as special functions, as they contained most
of the previously studied families of symmetric functions such as
Schur polynomials, Jack polynomials, Hall–Littlewood polynomials and
Askey–Wilson polynomials as special cases. They also satisfied many exciting
properties, among which we just mention one, which led to a
remarkable relation between Macdonald polynomials, representation
theory, and algebraic geometry. This property, called \emph{Macdonald’s
postivity conjecture} \cite{Macdonald1988}, states that the coefficients
$K^{(q,t)}_{\mu,\la}$ in the expansion of $J_\la^{(q,t)}(\bm{x})$ into the ``plethystic
Schur'' basis $s_\mu[\bm{X}(1-t)]$ (for the readers not familiar with
the plethystic notation we refer to \cite[Chapter VI.8]{Macdonald1995}) are polynomials in $q,t$ with
nonnegative integer coefficients. Garsia and Haiman \cite{GarsiaHaiman1993} refined this conjecture,
giving a representation theoretic interpretation for the coefficients in terms of Garsia-Haiman modules, an
interpretation which was finally proved almost ten years later by
Haiman \cite{Haiman2001}, who connected the problem to the
study of the Hilbert scheme of $N$ points in the plane from algebraic
geometry. It quickly turned out that Macdonald
polynomials have found applications in special function theory, representation theory, algebraic geometry, 
group theory, statistics, quantum mechanics, and much more
\cite{GarsiaRemmel2005}. Moreover, their fascinating and rich combinatorial structure
is one of the most important object of interest in contemporary
algebraic combinatorics.

\subsection{Strong factorization property of interpolation Macdonald
  polynomials}
\label{subsec:SFP}

The main goal of this paper is to state and partially prove a
generalization of the celebrated Macdonald's positivity conjecture. We
are going to do it by proving that Macdonald polynomials have a \emph{strong
  factorization property} when $q \to 1$, which also resolves the
problem posed by the author of this paper and F\'eray in
our recent joint paper \cite[Conjecture 1.5]{DolegaFeray2016}. 

In order to explain the notion of \emph{strong factorization
  property}, let us introduce a few notations.
If $\lambda$ and $\mu$ are partitions,
we denote $\lambda \oplus \mu := (\la_1+\mu_1,\la_2+\mu_2,\dots)$ their entry-wise sum;
see \cref{SubsecPartitions}.
If $\lambda^1,\dots,\lambda^r$ are partitions
and $I$ a subset of $[r]:=\{1,\cdots,r\}$,
then we denote 
\[\la^I:= \bigoplus_{i \in I} \la^i.\]
Moreover, we use a standard notation:

\begin{definition}
\label{def:Osymbol}
For $r \in R$, where $R$ is a ring and $f,g\in R(q)$,
we write $f=O_r(g)$ if the rational function 
$\frac{f(q)}{g(q)}$ has no pole in $q = r$.
\end{definition}

Then, we prove the following theorem:

\begin{theorem}
    For any partitions $\lambda^1,\dots,\lambda^r$, Macdonald
    polynomials have the {\em strong factorization property} when $q \to 1$, {\it i.e.}
    \begin{equation}
        \prod_{I \subset [r]} \left(J_{\la^I}^{(q,t)}\right)^{(-1)^{|I|}}
        =1+O_1\left((q-1)^{r-1}\right).
        \label{EqSFMacdo}
    \end{equation}
    \label{ThmSFJack}
\end{theorem}

As in our previous paper \cite{DolegaFeray2016}, let us unpack the notation for small values of $r$ in order to explain the terminology {\em strong factorization property}.
\begin{itemize}
    \item For $r=2$, \cref{EqSFMacdo} writes as
        \[ \frac{J^{(q,t)}_{\la^1 \oplus \la^2}}{J^{(q,t)}_{\la^1}J^{(q,t)}_{\la^2}}
        = 1+O_1\left(q-1\right). \]
        In other terms, this means that for $q=1$,
        one has the factorization property 
        $J^{(1,t)}_{\la^1 \oplus \la^2}=J^{(1,t)}_{\la^1}J^{(1,t)}_{\la^2}$.
        This is indeed true and follows from an explicit expression
        for $J^{(1,t)}_{\la}$ given in \cite[Chapter VI, Remark (8.4)-(iii)]{Macdonald1995}. 
        Thus, in this case, our theorem does not give anything new.
    \item For $r=3$, \cref{EqSFMacdo} writes as
        \[\frac{J^{(q,t)}_{\la^1 \oplus \la^2 \oplus \la^3} \, J^{(q,t)}_{\la^1} \, J^{(q,t)}_{\la^2}
        J^{(q,t)}_{\la^3}}{J^{(q,t)}_{\la^1 \oplus \la^2} J^{(q,t)}_{\la^1 \oplus \la^3} J^{(q,t)}_{\la^2 \oplus \la^3}}
                = 1+O_1\left((q-1)^2\right). \]
        Using the case $r=2$, it is easily seen that the left-hand side is $1+O_1\left(q-1\right)$.
        But our theorem says more and asserts that it is $1+O_1\left((q-1)^2\right)$,
        which is not trivial at all.
\end{itemize}
\medskip

\cref{ThmSFJack} has an equivalent form that uses the notion of {\em
  cumulants of Macdonald polynomials} (see \cref{sec:cumulants} for comments on the terminology).
For partitions $\la^1,\cdots,\la^r$, we denote
\[\ka^J(\la^1,\cdots,\la^r) :=
   \sum_{\substack{\pi \in \PPP([r]) } } 
   (-1)^{\#\pi}\prod_{B \in \pi} J^{(q,t)}_{\la^B}.
\]
Here, the sum is taken over set partitions $\pi$ of $[r]$ and
$\#\pi$ denotes the number of parts of $\pi$;
see \cref{SubsecSetPartitions} for details. 
For example
\begin{align*}
 %  \kumu(J^{(q,t)}_{\lambda^1})    = & J^{(q,t)}_{\lambda^1}, \\ 
   \ka^J(\la^1,\la^2) & = J^{(q,t)}_{\lambda^1\oplus\lambda^2}- 
                                      J^{(q,t)}_{\lambda^1}J^{(q,t)}_{\lambda^2}, \\ 
   \ka^J(\la^1,\la^2,\la^3) & = J^{(q,t)}_{\lambda^1\oplus\lambda^2\oplus\lambda^3}- 
                                      J^{(q,t)}_{\lambda^1}J^{(q,t)}_{\lambda^2\oplus\lambda^3}\\  & \qquad -
                                      J^{(q,t)}_{\lambda^2}J^{(q,t)}_{\lambda^1\oplus\lambda^3}-
                                      J^{(q,t)}_{\lambda^3}J^{(q,t)}_{\lambda^1\oplus\lambda^2} + 2J^{(q,t)}_{\lambda^1}J^{(q,t)}_{\lambda^2}J^{(q,t)}_{\lambda^3}.
\end{align*} 
An equivalent form of
\cref{ThmSFJack} in terms of cumulants is as follows:
\begin{theorem}
\label{conj:Macdonald}
For any partitions $\la^1, \dots, \la^r$, Macdonald polynomials have a
{\em
  small cumulant property} when $q \to 1$, that is
    \begin{equation*}
        \kumu^J(\la^1,\cdots,\la^r)=O_1\left( (q-1)^{r-1} \right).
        % \label{EqSCMacdo}
    \end{equation*}
\end{theorem}

Instead of proving \cref{conj:Macdonald}, we prove the stronger result
that \emph{interpolation Macdonald polynomials} have a small cumulant property when $q \to 1$, from which \cref{conj:Macdonald} follows as a
special case. To make this section complete, let us introduce
interpolation Macdonald polynomials.

\medskip

Interpolation polynomials are characterized
by certain vanishing condition. Sahi \cite{Sahi1996} proved that for each partition $\la$ of length
$\ell(\la) \leq N$, there exists a unique (inhomogenous) symmetric
\sloppy polynomial $\J^{(q,t)}_\la(\bm{x})$ of degree $|\la|$, where $\bm{x} =
(x_1,\dots,x_N)$, which has the following properties:
\begin{itemize}
\item
in the monomial basis expansion the coefficient
$[m_\la]\J^{(q,t)}_\la(\bm{x})$ is the same as $ [m_\la]J^{(q,t)}_\la(\bm{x})$;
\item
 for all partitions $\mu \neq \la, |\mu| \leq |\la|$ an expression $\J^{(q,t)}_\la(\widetilde{\mu})$ vanishes, where 
\[ \widetilde{\mu}  := (q^{\mu_1}t^{N-1},q^{\mu_2}t^{N-2},\dots,q^{\mu_N}t^{0}
).\]
\end{itemize}
This symmetric polynomial is called \emph{interpolation Macdonald
  polynomial} and it has a remarkable property which explains its name: its top-degree
part is equal to Macdonald polynomial $J^{(q,t)}_\la(\bm{x})$.

Our main result is the following theorem:

\begin{theorem}
\label{theo:CumulantsMacdonald}
    Let $\lambda^1,\cdots,\lambda^r$ be partitions.
    Then
we have a following small cumulant property when $q \to 1$:
\begin{equation*}
    % \label{eq:AsymptFactoMacdonald}
    \kumu^{\J}(\la^1,\cdots,\la^r) = O_1\left( (q-1)^{r-1} \right),
\end{equation*}
where $\kumu^{\J}(\la^1,\cdots,\la^r)$ is a cumulant of interpolation
Macdonald polynomials.
\end{theorem}

Since the top-degree part of $\kumu^{\J}(\la^1,\dots,\la^r)$ is equal
to $\kumu^J(\la^1,\dots,\la^r)$, \cref{conj:Macdonald} follows.

\subsection{Higher--order Macdonald's positivity conjecture}

As we already mentioned, the purpose of this paper is to generalize
$q,t$-Kostka numbers and to prove that they are polynomials in $q,t$
with integer coefficients. Before we define the \emph{multivariate $q,t$-Kostka
numbers} we just mention that strictly from the definition of $q,t$-Kostka numbers, they are
elements of $\QQ(q,t)$, and it took six or seven years after
Macdonald formulated his conjecture to prove that they are in fact
polynomials in $q,t$ with integer coefficients, which was proved
independently by many authors \cite{GarsiaRemmel1998,
  GarsiaTesler1996, KirillovNoumi1998, Knop1997,
  LapointeVinet1997, Sahi1996}. This result will be important to prove
the integrality of the multivariate $q,t$-Kostka numbers.

Let $\lambda^1,\cdots,\lambda^r$ be partitions. We define the
\emph{multivariate $q,t$-Kostka numbers}
$K^{(q,t)}_{\mu; \la^1,\dots,\la^r}$ by the following equation
\[ \kumu^J(\la^1,\dots,\la^r) = (q-1)^{r-1}
\sum_{\mu\ \vdash \left|\la^{[r]}\right|}K^{(q,t)}_{\mu;\la^1,\dots,\la^r}s_\mu[\bm{X}(1-t)]. \]

Note that when $r=1$, the multivariate $q,t$-Kostka number
$K^{(q,t)}_{\mu; \la^1}$ is equal to the ordinary $q,t$-Kostka number
$K^{(q,t)}_{\mu,\la}$ with $\la^1 = \la$.

In particular, integrality of Littlewood-Richardson
coefficients together with the integrality result on
$q,t$-Kostka numbers implies that
\[ (q-1)^{r-1} K^{(q,t)}_{\mu;\la^1,\dots,\la^r} \in \Z[q,t].\]

Thus applying \cref{conj:Macdonald} into the above result, we obtain immediately the following theorem:

\begin{theorem}
\label{theo:GeneralizedKostka}
Let $\lambda^1,\dots,\lambda^r$ be partitions. Then, for any
partition $\mu$, the multivariate $q,t$-Kostka number
$K^{(q,t)}_{\mu; \la^1,\dots,\la^r}$ is a polynomial in $q,t$ with
integer coefficients.
\end{theorem}

We recall that Macdonald's positivity conjecture is a well-established
theorem nowadays since Haiman proved it in 2001
\cite{Haiman2001}. We ran some computer
simulations which suggested that multivariate $q,t$-Kostka numbers are also
polynomials with positive coefficients. Unfortunately, we are not able to prove it, since our techniques of
the proof of \cref{theo:CumulantsMacdonald} does
not seem to be applicable to this problem and we state it in this paper as a conjecture.

\begin{conjecture}
Let $\lambda^1,\dots,\lambda^r$ be partitions. Then, for any
partition $\mu$, the multivariate $q,t$-Kostka number
$K^{(q,t)}_{\mu;\la^1,\dots,\la^r}$ is a polynomial in $q,t$ with
positive, integer coefficients.
\end{conjecture}

\subsection{Related problems}
We finish this section, mentioning some similar or somewhat related problems.
First, we recall that one of the most typical
application of cumulant is to show that a certain family of random
variables is asymptotically Gaussian. Especially, when one deals with
discrete structures, the main
technique is to show that cumulants have a certain \emph{small cumulant
  property}, which is in the same spirit as our \cref{conj:Macdonald};
see \cite{Sniady2006c, FerayMeliot2012, Feray2013,
  DolegaSniady2017}. It is therefore natural to ask for a
probabilistic interpretation of \cref{conj:Macdonald}. In
particular, does it
lead to some kind of central limit theorem? The most natural framework
to investigate this problem seems to be related with Macdonald
processes introduced by Borodin and Corwin \cite{BorodinCorwin2014} or
representation-theoretical interpretation of Macdonald polynomials
given by Haiman \cite{Haiman2001}.

A second problem is related to the combinatorics of Jack polynomials,
which are special cases of Macdonald polynomials. In fact,
\cref{conj:Macdonald} was posed as an open question in our previous
paper joint with F\'eray \cite{DolegaFeray2016}, where we proved that
Jack polynomials have a strong factorization property when $\alpha \to 0$, where
$\alpha$ is the Jack-deformation parameter. In the same paper we use this result as a key tool to
prove the polynomiality part of the so-called \emph{$b$-conjecture}, stated by
Goulden and Jackson \cite{GouldenJackson1996}. This conjecture says
that a certain
multivariate generating function involving Jack symmetric functions
expressed in the power-sum basis gives rise to the multivariate generating function of bipartite maps
(bipartite graphs embedded into some surface), where the exponent  of $\beta
:= \alpha - 1$ 
has an interpretation as some mysterious ``measure of
non-orientability'' of the associated map. The conjecture is still open,
while some special cases have been solved \cite{GouldenJackson1996,
  LaCroix2009, KanunnikovVassilieva2014, Dolega2016}. It is very
tempting to build a $q,t$-framework which will generalize the
$b$-conjecture. Although we can simply replace Jack polynomials
by Macdonald polynomials in the definition of the multivariate generating
function given by Goulden and Jackson and use the same techniques as in
\cite{DolegaFeray2016} to prove that expanding it in a properly
normalized power-sum basis
we obtain polynomials in $q,t$, we \emph{do not} obtain positive,
neither integer coefficients. Therefore, we leave wide-open a question
of the possibility of building a proper framework which generalizes
the $b$-conjecture to two parameters in a way that it is related to
counting some combinatorial objects.

\subsection{Organization of the paper}
\label{subsec:outline}

We describe all necessary definitions and background in
\cref{sec:preliminaries}. \cref{sec:Proof} gives the proof of
\cref{theo:CumulantsMacdonald} which is preceded by an explanation of the
main idea of the proof. In \cref{sec:cumulants} we
discuss cumulants and their relation with the strong factorization
property, and we investigate a relation between cumulants and
derivatives that is in the heart of the proof of \cref{theo:CumulantsMacdonald}.
Finally, \cref{sec:proof} is devoted to the proof of two intermediate
steps of the proof of \cref{theo:CumulantsMacdonald}.

\section{Preliminaries}
\label{sec:preliminaries}

\subsection{Set partitions lattice}
\label{SubsecSetPartitions}
The combinatorics of set partitions is central in the theory of cumulants 
 and will be important in this article.
 We recall here some well-known facts about them.

A {\em set partition} of a set $S$ is a (non-ordered) family of non-empty disjoint
subsets of $S$ (called parts of the partition), whose union is $S$.
In the following, we always assume that $S$ is finite.

Denote $\PPP(S)$ the set of set partitions of a given set $S$.
Then $\PPP(S)$ may be endowed with a natural partial order:
the {\em refinement} order.
We say that $\pi$ is {\em finer} than $\pi'$ (or $\pi'$ {\em coarser} than $\pi$)
if every part of $\pi$ is included in a part of $\pi'$.
We denote this by $\pi \leq \pi'$.

Endowed with this order, $\PPP(S)$ is a complete lattice, which means that
each family $F$ of set partitions admits a join (the finest set partition
which is coarser than all set partitions in $F$; we denote the join operator by $\vee$)
and a meet (the coarsest set partition
which is finer than all set partitions in $F$; we denote the meet operator by $\wedge$).
In particular, the lattice $\PPP(S)$ has a maximum $\{S\}$ (the partition in only one
part) and a minimum $\{ \{x\}, x \in S\}$ (the partition in singletons).

% Moreover, this lattice is ranked:
% the rank $\rk(\pi)$ of a set partition $\pi$ is $|S|-\#\pi$,
% where $\#\pi$ denotes the number of parts of $\pi$.
% The rank is compatible with the lattice structure in the following sense:
% for any two set partitions $\pi$ and $\pi'$,
% \begin{equation*}% \label{EqRkJoin}
% \rk(\pi \vee \pi') \leq \rk(\pi) + \rk(\pi').
% \end{equation*}

Lastly, denote $\mu$ the M\"obius function of the partition lattice $\PPP(S)$.
Then, for any pair $\pi \leq \si$ of set partitions, the value of the M\"obius function has a product form:
\begin{equation}\label{EqValueMobius}
    \mu(\pi, \si)=\prod_{B' \in \si}\mu\left(\{B \in \pi: B \subset B'\}, \{B'\}\right),
\end{equation}
where the product is taken over all blocks of a partition $\si$, and
for a given block $B' \in \si$ an expression $\mu\left(\{B \in \pi: B \subset B'\}, \{B'\}\right)$ 
denotes a M\"obius
function of the lattice $\PPP(B')$ of the interval in between a
partition $\{B \in \pi: B \subset B'\}$, and a maximal element
$\{B'\}$. This function is given by an explicit formula
\[ \mu\left(\{B \in \pi: B \subset B'\}, \{B'\}\right) = (-1)^{\#\{B
  \in \pi: B\subset B'\}-1} \left(\#\{B \in \pi: B\subset
  B'\}-1\right)!,\]
where $\#\pi$ denotes the number of parts of $\pi$.

We finish this section by stating a well-known result on computing a M\"obius functions of lattices.

\begin{proposition}[Weisner's Theorem \cite{Weisner1935}]
\label{prop:Weisner}
For any $\pi < \tau \leq \si$ in a lattice $L$ we have
\[ \sum_{\substack{\pi \leq \omega \leq \si: \\ \omega \vee \tau = \si}}\mu(\pi, \omega) = 0.\]
\end{proposition}

\subsection{Partitions}
\label{SubsecPartitions}
We call $\lambda := (\lambda_1, \lambda_2, \dots, \lambda_l)$ \emph{a partition} of $n$
if it is a weakly decreasing sequence of positive
integers such that $\la_1+\la_2+\cdots+\la_l = n$.
Then $n$ is called {\em the size} of $\lambda$ while $l$ is {\em its length}.
As usual we use the notation $\la \vdash n$, or $|\la| = n$, and $\ell(\la) = l$.
We denote the set of partitions of $n$ by $\Y_n$ and we define a partial order on $\Y_n$,
called {\em dominance order}, in the following way:
\[ \lambda \leq \mu \iff \sum_{i\leq j}\lambda_i \leq \sum_{i\leq j}\mu_i \text{ for any positive integer } j.\]
Then, we extend the notion of dominance order on the set of partitions of arbitrary size by saying that
\[ \lambda \preceq \mu \iff |\la| < |\mu|, \text{ or } |\la| = |\mu|
\text{ and } \la \leq \mu.\]

For any two partitions $\la \in \Y_n$ and $\mu \in \Y_m$ we can
construct a new partition $\la \oplus \mu \in \Y_{n+m}$ by setting $\la \oplus \mu := (\la_1+\mu_1,\la_2+\mu_2,\dots)$.
% and $\la \cup \mu$ is obtained by merging parts of $\la$ and $\mu$ and ordering them in a decreasing fashion.
Moreover, there exists a canonical involution on the set $\Y_n$,
which associates with a partition $\la$ its \emph{conjugate partition} $\la^t$.
By definition, the $j$-th part $\la_j^t$ of the conjugate partition
is the number of positive integers $i$ such that $\la_i \ge j$.
%Notice that for any two partitions $\la, \mu$, we have 
%$(\la \cup \mu)^t=\la^t \oplus \mu^t$.
A partition $\la$ is identified with some geometric object, called \emph{Young diagram},
that can be defined as follows:
\[ \la = \{(i, j):1 \leq i \leq \la_j, 1 \leq j \leq \ell(\la) \}.\]
 For any
 box $\square := (i,j) \in \la$ from Young diagram we define its \emph{arm-length} by $a(\square) := \la_j-i$ and its \emph{leg-length} by $\ell(\square) := \la_i^t-j$ (the same definitions as in \cite[Chapter I]{Macdonald1995}), see \cref{fig:ArmLeg}.

    \begin{figure}[t]
        \[
    \begin{array}{c}
    \Yfrench
    \Yboxdim{1cm}
    \begin{tikzpicture}[scale=0.75]
        \newcommand\yg{\Yfillcolour{gray!40}}
        \newcommand\yw{\Yfillcolour{white}}
        \newcommand\tralala{(i,j)}
        \newcommand\arml{a(\square)}
        \newcommand\legl{\ell(\square)}
        \tgyoung(0cm,0cm,;;;;;;;;;;,;;_6,;|4;;;;,;:;;;,;:;;,;:;)
       % \Yfillopacity{0}
       % \Ylinethick{0.3pt}
       \Yfillopacity{1}
       \Yfillcolour{gray!40}
       \tgyoung(2cm,1cm,_6\arml)
       \tgyoung(1cm,2cm,|4\legl)
       \Yfillopacity{0}
       \Ylinethick{2pt}
       \tgyoung(1cm,1cm,\square)
       \draw[->](1.5,3.5)--(1.5,2.3);
       \draw[->](1.5,4.5)--(1.5,5.7);
       \draw[->](5.7,1.5)--(7.7,1.5);
       \draw[->](4.3,1.5)--(2.3,1.5);
       \end{tikzpicture}
    \end{array}
        \]
     \caption{Arm and leg length of boxes in Young diagrams.}
     \label{fig:ArmLeg}
    \end{figure}
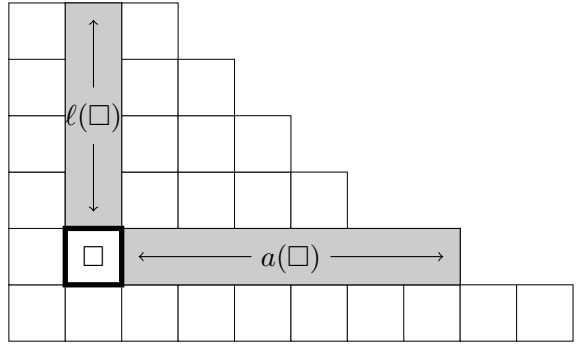

Finally, we define two combinatorial quantities associated with partitions that we will use extensively through this paper.
First, 
we define the \emph{$(q,t)$-hook polynomial} $h_{(q,t)}(\la)$ 
%and $h'_{(q,t)}(\la)$ 
by the following equation
\begin{align}
\label{eq:HookProduct}
h_{(q,t)}(\la) &:= \prod_{\square \in \la}\left(1 - q^{a(\square)}t^{\ell(\square) +1} \right).
%\\
%\label{eq:HookProduct2}
%h'_{(q,t)}(\la) &:= \prod_{\square \in \la}\left(1 - q^{a(\square)+1}t^{\ell(\square) } \right).
\end{align}
We also introduce a partition binomial given by
\begin{equation}
\label{eq:BinomialEigenvalue}
b^N_j(\lambda) := \sum_{1 \leq i \leq N}\binom{\lambda_i}{j}t^{N-i}.
\end{equation}

\subsection{Interpolation Macdonald polynomials as eigenfunctions}
\label{subsec:laplace}

We already defined interpolation Macdonald polynomials in
\cref{subsec:SFP}, but we are going to introduce another, equivalent
definition that is more convenient in the framework of the following
paper. Since this is now a well-established theory,
results of this section are given without proofs but with explicit references to the literature
(mostly to Macdonald's book \cite{Macdonald1995} and Sahi's paper \cite{Sahi1996}).

First, consider the vector space $\Symm_N$ of symmetric polynomials in
$N$ variables over $\QQ(q,t)$. Let $T_{q,x_i}$ be the ``q-shift
operator'' defined by 
\[T_{q,x_i} f(x_1,\dots,x_N) := f(x_1,\dots,qx_i,\dots,x_N),\]
and 
\[A_i(\bm{x};t)f(x_1,\dots,x_N) := \left(\prod_{j\neq i} \frac{tx_i-x_j}{x_i-x_j}\right)f(x_1,\dots,x_N).\]
Let us define an operator 
\begin{equation} 
\label{eq:D}
D := \sum_i A_i(\bm{x};t)(1-x_i^{-1})\left(T_{q,x_i}-1\right).
\end{equation}
\begin{proposition}
    \label{PropDefMacdonald}
    There exists a unique family $\J_\la^{(q,t)}$ (indexed by partitions $\la$ of length at most $N$)
    in $\Symm^N$ that satisfies:
    \begin{enumerate}[noitemsep,topsep=0pt,parsep=0pt,partopsep=0pt]
        \item[(C1)]
            $\J^{(q^{-1},t^{-1})}_\la(\bm{x})$ is an eigenvector of $D$
with eigenvalue 
\[\ev(\la) := \sum_{1 \leq i \leq N} (q^{\la_i}-1)t^{N-i} = \sum_{j
  \geq 1}(q-1)^j b^N_j(\lambda);\]
\item[(C2)] the monomial expansion of $\J_\la^{(q,t)}$ is given by
  \begin{equation*}
   \J^{(q,t)}_\la = h_{(q,t)}(\la) m_\la + \sum_{\nu \prec
     \la}a^{\la}_\nu m_\nu,  \text{ where } a^{\la}_\nu
     \in \begin{cases} \Z[q,t] &\text{ for } |\nu| = |\la|,\\ \Z[q,t^{-1},t] &\text{ for } |\nu| < |\la|.\end{cases}
  \end{equation*}
    \label{eq:MacdonaldInMonomial}
    \end{enumerate}
    These polynomials are called {\em interpolation Macdonald polynomials}.
\end{proposition}
This is a result of Sahi \cite{Sahi1996}. His original
definition requires that the coefficients $a^{\la}_\nu$ are only
rational functions in $q,t$ with rational coefficients, but in the
same paper Sahi proved that they are in fact polynomials in
$q,t^{-1},t$ (and even in $q,t$ when $|\nu| = |\la|$) with
integer coefficients, which will be important for us later. We just add for completness of the presentation that
we are using different notation and normalization than Sahi, so
function $R_\la(x;q^{-1},t^{-1})$ from Sahi's paper \cite{Sahi1996} is equal to 
$\left( h_{(q,t)}(\la)\right)^{-1}\J^{(q,t)}_\la(\bm{x})$ with our notation, and $c_\la(q,t)$ from Sahi's paper is the same as
$h_{(q,t)}(\la)$ with our notation.

Above definition says that the interpolation
Macdonald polynomial $\J_\la^{(q,t)}$ depends on the parameter $N$,
that is the number of variables.
However, one can show that it satisfies the compatibility relation
$\J_\la^{(q,t)}(x_1,\dots,x_N,0)=\J_\la^{(q,t)}(x_1,\dots,x_N)$ and thus $\J_\la^{(q,t)}$ can be seen as a symmetric function.
In the sequel, when working with differential operators,
we sometimes confuse a symmetric function $f$
with its restriction $f(x_1,\dots,x_N,0,0,\dots)$ to $N$ variables.

\medskip

It was shown by Macdonald \cite[Chapter VI, (3.9)--(3.10)]{Macdonald1995} that
\[ \left(\sum_{1 \leq i \leq N}A_i(\bm{x};t)T_{q,x_i}\right) m_\la = \left(\sum_{1 \leq i \leq N}q^{\la_i}t^{N-i}\right) m_\la + \sum_{\nu < \la}b^{\la}_\nu m_\nu,\]  
where $b^{\la}_\nu \in \Z[q,t]$. Moreover, it is easy to show (see for example \cite[Lemma 3.3]{Sahi1996}) that
\begin{equation*}
% \label{eq:1stRelation}
\sum_i A_i(\bm{x};t) = \sum_i t^{N-i}.
\end{equation*}
Plugging it into \cref{eq:D} we observe that:
\[ D\ m_\la = \ev(\la)\ m_\la + \sum_{\nu \prec \la}c^{\la}_\nu\ m_\nu,
\text{ where } c^{\la}_\nu \in \Z[q,t].\]
Note that we can expand operator $D$ around $q=1$ as a linear combination of differential operators in the following form:
\begin{equation} 
\label{eq:DAsDifferential}
D = \sum_{j \geq 1}\frac{(q-1)^{j}}{j!}\sum_i
\left(A_i(\bm{x};t)(x_i^j-x_i^{j-1}) D^j_i\right),
\end{equation}
where $D^j_i := \frac{\partial^j}{\partial x_i^j}$. 
As a consequence we have the following identity:
\begin{multline}
\label{eq:MonomialAction}
\sum_{1 \leq i \leq N}\left(A_i(\bm{x};t)(x_i^j-x_i^{j-1})D^j_i
\right) m_\la = \partial_q^{j}\big(\ev(\la)\big)_{q=1}
m_\la \\
+ \sum_{\nu \prec
   \la}\partial^{j}_q \left(c^{\la}_\nu\right)_{q=1}
 m_\nu = j!\ b^N_j(\la)\ m_\la + \sum_{\nu \prec \la}d^{\la}_\nu m_\nu,
\end{multline}
where $\partial_q$ is a partial derivative with respect to $q$, $b^N_j(\la)$ is given by \cref{eq:BinomialEigenvalue}, and $d^{\la}_\nu \in \Z[t]$.

\begin{corollary}
\label{cor:EigenvectorsVanish}
Let $f \in \Symm$ be a symmetric function with an expansion in the monomial basis of the following form:
\[ f = \sum_{\mu \prec \lambda} d_\mu m_\mu,\]
where $\la$ is a fixed partition, and $d_\mu \in \QQ(t)$.
If, for any number $N$ of variables, $\sum_{1 \leq i \leq N}\left(A_i(\bm{x};t)(x_i-1)D^1_i \right) f = b^N_1(\la) f$ then $f = 0$.
\end{corollary}

\begin{proof}
  It is obvious from \cref{eq:MonomialAction} since $b^N_1(\la) = b^N_1(\mu)$ implies $\la = \mu$.
\end{proof}

\section{Strong factorization property of interpolation Macdonald
  polynomials}
\label{sec:Proof}

In this section we prove \cref{theo:CumulantsMacdonald}. Since its
proof involves many intermediate results which can be considered as
independent of \cref{theo:CumulantsMacdonald}, we believe that presenting them before the proof of the main result might
discourage the reader, and we decided to explain the main idea of the
proof of \cref{theo:CumulantsMacdonald} first, then give the proof
with all the details, and finally present all the remaining proofs of the
intermediate results in the separate sections.

\begin{proof}[Proof of \cref{theo:CumulantsMacdonald}]
We recall that we need to prove that for any positive
integer $r$, and for any partitions $\la^1,\dots,\la^r$ we have the
following bound for the cumulant:
\[ \kumu^{\J}(\la^1,\dots,\la^r) = O_1\left((q-1)^{r-1}\right).\]
The proof will by given by induction on $r$.
The fact that Macdonald interpolation polynomials $\J^{(q,t)}_\la$ have
no singularity in $q=1$ is straightforward from the result
of Sahi presented in \cref{PropDefMacdonald}. That covers the case
$r=1$.

Now, notice that for any ring $R$, and any rational function $f \in
R[q]$, the following conditions are equivalent
\[ f(q) = O_1\left((q-1)^r\right) \iff  f(q^{-1}) =
O_1\left((q-1)^r\right).\]
Thus, we are going to prove that
\[ \kumu^{\J}(\la^1,\dots,\la^r) = O_1\left((q-1)^r\right),\]
where $\kumu^{\J}(\la^1,\dots,\la^r)$ denotes the cumulant with parameters
$q^{-1},t^{-1}$. From now on, until the end of this proof, $\kumu^{\J}(\la^1,\dots,\la^r)$ denotes the cumulant with parameters
$q^{-1},t^{-1}$.

\medskip

Let $R$ be a ring, and let $f \in R[q,q^{-1}]$ be a Laurent polynomial
in $q$. We introduce the
following notation: for any nonnegative integer $k$ the coefficient
$[(q-1)^k]f \in R$ is defined by the following expansion:
\[ q^{\deg(f)}f = \sum_{k \geq 0}\bigg([(q-1)^k]f\bigg)\ (q-1)^k,\]
where $\deg(f)$ is the smallest possible nonnegative integer such that 
\[ q^{\deg(f)}f \in R[q].\]
It is clear that for two Laurent polynomials $f,g \in R[q,q^{-1}]$
and nonnegative integer $k$ one has the following identity:
\[ \left[(q-1)^k\right](fg) = \sum_{0 \leq j \leq k}\bigg(\big[(q-1)^j\big]f\bigg)\cdot \bigg(\big[(q-1)^{k-j}\big]g\bigg).\]
With the above notation, we have to prove that for any integer $0 \leq
k \leq r-2$ the following equality holds true:
\[ f := \left[(q-1)^k\right]\kumu^{\J}(\la^1,\dots,\la^r) = 0.\]
Notice now that the expansion of $f$ into the monomial basis involves only
the monomials $m_\mu$ indexed by partitions $\mu \prec \la^{[r]}$, which
is ensured by \cref{cor:CumulantsTopMonomialExpansion}. Thus, if
we are able to show that the following equation holds true:
\begin{equation} 
\label{eq:DiffEquation}
\sum_{1 \leq i \leq N}A_i(\bm{x};t)(x_i - 1)D^1_i f =
b^N_1(\la^{[r]}) f,
\end{equation}
then $f=0$ by \cref{cor:EigenvectorsVanish}, and the proof is over. So
our goal is to prove \cref{eq:DiffEquation}. In order to
do that we make the following observation: an interpolation
Macdonald polynomial $\J^{(q^{-1},t^{-1})}$ is an eigenfunction of the operator
$D$. Since the cumulant is a linear combination of products of
interpolation Macdonald polynomials
\[ \kumu^{\J}(\la^1,\dots,\la^r) :=
   \sum_{\substack{\pi \in \PPP([r]) } } 
    (-1)^{\#\pi} \prod_{B \in \pi} \J^{(q^{-1},t^{-1})}_{\la^B},\]
 it will be very convenient if the
action of $D$ on such a product will be given by the Leibniz
rule, that is
\begin{multline*} 
D\left(\J^{(q^{-1},t^{-1})}_{\la^1}\cdots
  \J^{(q^{-1},t^{-1})}_{\la^r}\right) \\
= \sum_{1 \leq k \leq r}\J^{(q^{-1},t^{-1})}_{\la^1}\cdots \left(D \J^{(q^{-1},t^{-1})}_{\la^k}
\right) \cdots \J^{(q^{-1},t^{-1})}_{\la^r}.
\end{multline*}
Unfortunately, it is not the case. However, the trick is to
decompose $D\kumu^{\J}(\la^1,\dots,\la^r)$ into two parts: the first
part is given by ``forcing'' the Leibniz rule for the action of $D$ on the
product of interpolation Macdonald polynomials, and the second part is
given by the difference between the proper action of $D$ on cumulant, and
between the forced version. To be more precise
\begin{equation} 
\label{eq:11}
D\kumu^{\J}(\la^1,\dots,\la^r) =
\underbrace{\widetilde{D}\kumu^{\J}(\la^1,\dots,\la^r)}_{\text{first part}} + \underbrace{D\kumu^{\J}(\la^1,\dots,\la^r) -
  \widetilde{D}\kumu^{\J}(\la^1,\dots,\la^r)}_{\text{second part}},
\end{equation}
where
\[ \widetilde{D}\kumu^{\J}(\la^1,\dots,\la^r) := \sum_{\substack{\pi \in \PPP([r]) } } 
    (-1)^{\#\pi} \widetilde{D}\left( \J^{(q^{-1},t^{-1})}_{\la^B}: B \in
      \pi\right),\]
and
\[ \widetilde{D}(f_1,\dots,f_r) := \sum_{1 \leq k \leq r}f_1\cdots \left(D f_k
\right) \cdots f_r.\]
This decomposition turned out to be crucial. Indeed,
\cref{lem:A1SimpleForm} ensures that the first part
can be expressed as a linear combination of products of cumulants of
less then $r$ elements, thus we can use an induction hypothesis to
analyze it. Similarly, \cref{lem:A2SimpleForm} states that the second
part can be given by an expression involving products of cumulants of
less then $r$ elements, and again, an inductive hypothesis can be used
to its analysis. Then, comparing the coefficient of
$(q-1)^k$ in the left hand side of \cref{eq:11} with the coefficient
of $(q-1)^k$ in the right hand side of \cref{eq:11} we obtain
\cref{eq:DiffEquation}. Let us go into details.
Expanding operator $D$ around $q=1$ (see \cref{eq:DAsDifferential}) we have that
\begin{multline} 
\label{eq:1}
\left[(q-1)^k\right]D \kumu^{\J}(\la^1,\dots,\la^r) \\
= \sum_{j \geq 1}\sum_{1 \leq i \leq N}A_i(\bm{x};t)(x^j_i - x^{j-1}_i)D^j_i \left(\left[(q-1)^{k-j}\right]\kumu^{\J}(\la^1,\dots,\la^r) \right).
\end{multline}
Moreover, applying \cref{lem:A1SimpleForm}, we have that
\begin{multline} 
\label{eq:3}
 \left[(q-1)^k\right]\left(\widetilde{D}\kumu^{\J}(\la^1,\dots,\la^r)\right) \\
= \sum_{j \geq 1}\sum_{\substack{\si \in \PPP([r])\\ \#\si\leq j}} \IE_j\left(\la^B: B \in \si\right)  \left[(q-1)^{k-j}\right]\left(\prod_{B \in \si}\kumu^{\J}(\la^i: i
  \in B)\right),
\end{multline}
where $\IE_j\left(\la^B: B \in \si\right)$ is a certain polynomial in
$t$ with integer coefficients (at this stage
we do not need to know its explicit form, but for the interested
Reader it is given by
\cref{eq:InclusionExclusion}) which has the following
form in the special case:
\[ \IE_1(\lambda) = b^N_1(\lambda).\]
% \todo{"InEx" is not defined at this stage of the paper. It's not clear that the reader has to look at Lemma 5.2 to get the required definition}
Finally, applying \cref{lem:A2SimpleForm}, we obtain the following
identity
\begin{multline} 
\label{eq:2}
 \left[(q-1)^k\right]\left(D \kumu^{\J}(\la^1,\dots,\la^r) -
   \widetilde{D}\kumu^{\J}(\la^1,\dots,\la^r)\right) \\
= -\sum_{j \geq 2}\sum_{1 \leq i \leq N}A_i(\bm{x};t)\frac{x_i^j-x_i^{j-1}}{j!}
\sum_{\substack{\pi \in \PPP([r]), \\ 2 \leq \#\pi \leq j}
}\sum_{\substack{\a \in \N_+^\pi, \\ |\a| = j} } \binom{j}{\a}  \\
\cdot\left[(q-1)^{k-j}\right]\left(\prod_{B \in \pi} D_i^{\a(B)}\kumu^{\J}(\la^b: b
  \in B)\right).
\end{multline}
Here, $\N_+^\pi$ denotes the set of functions $\a: \pi \to \N_+$ with
positive integer values, the symbol $|\a|$ is defined as 
\[ |\a| := \sum_{B \in \pi}\a(B),\] 
and 
\[ \binom{j}{\a} := \frac{j!}{\prod_{B \in \pi}\a(B)!}.\]
% \todo{$D^{\alpha_B}$ is not defined at this stage of the paper, same comment as above}
We recall that the right hand side (RHS for short) of \cref{eq:1} is equal to the sum
of the right hand sides of \cref{eq:3} and \cref{eq:2}. Let $k=1$. Then
the RHS of \cref{eq:1} is equal to
\[ \sum_{1 \leq i \leq N}A_i(\bm{x};t)(x_i - 1)D^1_i f,\]
the RHS of \cref{eq:3} is equal to 
\[ \IE_1\left(\la^{[r]}\right)f = b^N_1(\la^{[r]})f,\]
where $f = \left[(q-1)^0\right]\kumu^{\J}(\la^1,\dots,\la^r)$, and the RHS
of \cref{eq:2} vanishes. Thus, we have shown that \cref{eq:DiffEquation}
holds true for $f =
\left[(q-1)^0\right]\kumu^{\J}(\la^1,\dots,\la^r)$, which implies that
$f=0$. Now, we fix $K \leq r-2$, and we assume that 
\[[(q-1)^m]\kumu^{\J}(\la^1,\dots,\la^r) = 0\] 
holds true for all $ 0 \leq m < K$. We are going to show that \cref{eq:DiffEquation}
holds true for $f =
\left[(q-1)^K\right]\kumu^{\J}(\la^1,\dots,\la^r)$.
First, note that for $k=K+1$ the RHS of \cref{eq:1} simplifies to
\[ \sum_{1 \leq i \leq N}A_i(\bm{x};t)(x_i - 1)D^1_i f.\]
Moreover, from the induction hypothesis for each subset $I$ with
$\emptyset \subsetneq I \subsetneq [r]$, one has $\kumu^{\J}(\la^i: i
\in I) = O\left((q-1)^{|I|-1}\right)$. Thus, for any set partition $\pi \in \PPP([r])$ which has at least two parts, one has
\[ \prod_{B \in \pi} \left(D^{j_B}_i\kumu^{\J}(\la^b: b\in B) \right) = O\left((q-1)^{r-\#\pi}\right),\]
where the $j_B$ are any nonnegative integers ($D^0_i = \Id$ by
convention).
It implies that the RHS of \cref{eq:2} vanishes.
Finally, again by induction hypothesis, all the elements of the form $\left[(q-1)^{k-j}\right]\left(\prod_{B \in \si}\kumu^{\J}(\la^i: i
  \in B)\right)$ that appear in the RHS of \cref{eq:3} vanish
except $\left[(q-1)^K\right]\kumu^{\J}(\la^1,\dots,\la^r) = f$. Thus
the RHS of \cref{eq:3} simplifies to
\[ \IE_1\left(\la^{[r]}\right)f = b^N_1(\la^{[r]})f,\]
which proves that \cref{eq:DiffEquation} holds true.
The proof is completed.
\end{proof}

\section{Cumulants}
\label{sec:cumulants}

In this section we introduce cumulants and we investigate an action of
derivations on them, which is crucial in the proof of
\cref{theo:CumulantsMacdonald}. We also explain the connection between the strong factorization property and the small cumulant property, and we
present some applications of it relevant for our work. We begin with
some definitions.

\subsection{Partial cumulants}
\begin{definition}    
    Let $(u_I)_{I \subseteq J}$
    be a family of elements in a field,
    indexed by subsets of a finite set $J$.
    Then its {\em partial cumulant} is defined as follows.
        For any non-empty subset $H$ of $J$, set
        \begin{equation}
            \ka_H(\uu) = \sum_{\substack{\pi \in \PPP(H) } } 
            \mu(\pi, \{H\}) \prod_{B \in \pi} u_B,
        \label{EqDefCumulants}
    \end{equation}
\end{definition}
where $\mu$ is the M\"obius function of the set partition lattice; see \cref{SubsecSetPartitions}.

The terminology comes from probability theory.
Let $J = [r]$, and let $X_1,\dots,X_r$ be random variables with finite moments
defined on the same probability space.
Then define $u_I=\mathbb{E}(\prod_{i \in I} X_i)$,
where $\mathbb{E}$ denotes the expected value.
The quantity $\ka_{[\r]}(\uu)$ as defined above,
is known as the {joint (or mixed) cumulant} of the random variables
$X_1,\dots,X_r$.
Also, $\ka_{H}(\uu)$ is the joint/mixed cumulant of the smaller family
$\{X_h, h \in H\}$.

Joint/mixed cumulants have been studied by Leonov and Shiryaev
in \cite{LeonovShiryaev1959} (see also an older note
of Schützenberger \cite{Schutzenberger1947}, where they are introduced
under the French name {\em déviation d'indépendance}).
They now appear 
in random graph theory \cite[Chapter 6]{JansonLuczakRucinski2000}
and have inspired a lot of work
in noncommutative probability theory \cite{Novak_Sniady:What_is_free_cum}.

A classical result -- see, {\em e.g.},
\cite[Proposition 6.16 (vi)]{JansonLuczakRucinski2000} --
is that \cref{EqDefCumulants} can be inverted as follows:
for any non-empty subset $H$ of $J$, 
\begin{equation}
   u_H = \sum_{\substack{\pi \in \PPP(H) } } 
 \prod_{B \in \pi} \ka_B(\uu).
\label{EqMomCum}
\end{equation}

\subsection{Derivations and cumulants}

Let $R$ be a ring. We define an $R$-module of derivations $\Der_K$
which consists of linear maps $D: R \to R$ satisfying the
following Leibniz rule:
\[ D(f\cdot g) = (Df)\cdot g + f \cdot (D g).\]
For any positive integers $r,k$, and for any elements $f_1,\dots,f_r
\in R$ we define
\[ \widetilde{D^k}(f_1,\dots,f_r) := \sum_{1 \leq i \leq r}f_1\cdots \left(D^k f_i
\right) \cdots f_r,\]
Let $K$ be a field, and $D \in \Der_K$ be a derivation. Then, for any family $\uu = (u_I)_{I \subseteq [r]}$
of elements in a field $K$ we define the following deformed action of
$D^k$ on the cumulant:
\[ \widetilde{D^k}\ \kumu_{[r]}(\uu) := \sum_{\substack{\pi \in \PPP([r]) } } 
            \mu(\pi, \{[r]\})\ \widetilde{D^k}(u_B: B \in \pi).\]

The following lemma will be crucial to prove our main result.

\begin{lemma}
\label{lem:zasada}
For any positive integers $r,k$, for any family $\uu = (u_I)_{I \subseteq [r]}$
of elements in a field $K$ and for any derivation $D \in \Der_K$, the following identity holds true:
\begin{equation}
\label{eq:CumulantDifferential}
\widetilde{D^k}\ \kumu_{[r]}(\uu) = \sum_{\substack{\pi \in \PPP([r]), \\ \#\pi \leq k} }\sum_{\substack{\a \in \N_+^\pi, \\ |\a| = k} } \binom{k}{\a} \prod_{B \in \pi} \left(D^{\a(B)}\kumu_{B}(\uu)\right).
\end{equation}
Here, $\N_+^\pi$ denotes the set of functions $\a: \pi \to \N_+$, the symbol $|\a|$ is defined as 
\[ |\a| := \sum_{B \in \pi}\a(B),\] 
and 
\[ \binom{k}{\a} := \frac{k!}{\prod_{B \in \pi}\a(B)!}.\]
\end{lemma}

\begin{proof}
First of all, notice that for any elements $f_1,\dots,f_r \in K$, and for any positive integer $k$ the following generalized Leibniz rule holds true:
\begin{equation} 
\label{eq:GeneralLeibniz}
D^k(f_1\cdots f_r) = \sum_{\substack{\a \in \N^{[r]}, \\ |\a| = k} } \binom{k}{\a}\left(D^{\a(1)}f_1\right)\cdots\left(D^{\a(r)}f_r\right),
\end{equation}
which is easy to prove by induction ($D^0 := \Id$ by convention).

Notice now that the both hands of \cref{eq:CumulantDifferential} are linear combinations of elements of the form
\[ \prod_{B \in \pi}D^{\a(B)}u_B,\]
where $\pi \in \PPP([r])$, and $\a \in \N^\pi$ is a composition of $k$.
Let us call $\RHS$ the right-hand side of \cref{eq:CumulantDifferential}, and analogously $\LHS$ the left-hand side of \cref{eq:CumulantDifferential}. Let us fix a set partition $\pi \in \PPP([r])$, and a composition $\a \in \N^\pi$ of $k$. We would like to show that
\[ \left[\prod_{B \in \pi}D^{\a(B)}u_B\right] \LHS = \left[\prod_{B \in \pi}D^{\a(B)}u_B\right] \RHS.\]

We define the support $\supp(\a)$ of $\a$ in a standard way:
\[ \supp(\a) := \{B \in \pi: \a(B) \neq 0\}.\] 
Then, it is clear from definition of $\widetilde{D^k}\
\kumu_{[r]}(\uu)$ that
\begin{equation}
\label{eq:LHS}
\left[\prod_{B \in \pi}D^{\a(B)}u_B\right] \LHS = \begin{cases}\mu(\pi,\{[r]\}) &\text{ if } \#\supp(\a) = 1,\\ 0 &\text{ otherwise. }\end{cases}
\end{equation}

We now analyze the coefficient
\[ \left[\prod_{B \in \pi}D^{\a(B)}u_B\right] \RHS.\]
We can see that the nonzero contribution come from the elements of the following form:
\[ \prod_{B' \in \si} D^{\a'(B')}\kumu_{B'}(\uu), \]
where 
\[\a'(B') := \sum_{\substack{B \in \pi:\\ B \subset B'}}\a(B),\]
$\si \geq \pi$, and for each element $B' \in \si$ there exists an element $B \in \supp(\a)$ such that $B\subset B'$.
In other terms, $\si$ is a partition which has the property that $\si \geq \pi$, and
\[ \si \vee \tau = \{[r]\},\]
where partition $\tau$ is constructed from $\pi$ by merging all its blocks lying in a support of $\a$, i.e.~:
\begin{equation} 
\label{eq:tau}
\tau := \left\{\bigcup\supp(\a)\right\}\cup\left(\pi\setminus\supp(\a)\right).
\end{equation}

Using the definition of cumulants \cref{EqDefCumulants}, and \cref{eq:GeneralLeibniz}, we can compute the coefficient
%\[ \kumu_{B}(\uu) := \sum_{\substack{\pi \in \PPP(B) } } \mu(\pi, \{B\}) \prod_{B' \in \pi} u_B,\]
%thus combining it with \cref{eq:GeneralLeibniz}, we have that
%\[ \left[\prod_{B \in \pi}D^{\a(B)}u_B\right]\prod_{B' \in \si} D^{\a'(B')}\kumu_{B'}(\uu) = \]
%Then
%\[ \prod_{B \in \si} D^{\sum_{B' \in \pi; B' \subset B}\a_{B'}}\kumu_{B}(\uu) = \prod_{1 \leq i \leq \#(\tau)} D^{\sum_{j \in \tau_i}\a_{B_j}}\kumu_{\si_i}(\uu),\]
%and using \cref{eq:GeneralLeibniz}, we have that
\begin{multline*} 
\left[\prod_{B \in \pi}D^{\a(B)}u_B\right]\prod_{B' \in \si} D^{\a'(B')}\kumu_{B'}(\uu) \\
= \prod_{B' \in \si}\binom{\a'(B')}{\a(B): B \in \pi, B \subset B'}\mu\left(\{B \in \pi: B \subset B'\}, \{B'\}\right).
\end{multline*}
%In particular
%\begin{multline*} 
%\left[\prod_{B \in \pi}D^{\a_B}u_B\right] \binom{k}{\sum_{j \in \tau_i}\a_{B_j}: 1 \leq i \leq \#(\tau)}\prod_{1 \leq i \leq \#(\tau)} D^{\sum_{j \in \tau_i}\a_{B_j}}\kumu_{\si_i}(\uu) \\
%= (-1)^{\#\pi-\#(\tau)} \binom{k}{\a_B: B \in \pi}\prod_{1 \leq i \leq \#(\tau)}\left(|\tau_i| + |D_i|-1 \right)!.
%\end{multline*}
Plugging it into \cref{eq:CumulantDifferential}, we obtain that
\begin{multline*} 
\left[\prod_{B \in \pi}D^{\a(B)}u_B\right] \RHS = \sum_{\substack{\si \geq \pi,\\ \si \vee \tau = \{[r]\}}}\binom{k}{\a'(B'): B' \in \si}\left[\prod_{B \in \pi}D^{\a(B)}u_B\right]\prod_{B' \in \si} D^{\a'(B')}\kumu_{B'}(\uu)\\
= \binom{k}{\a}\sum_{\substack{\si \in \PPP([r]),\\ \si \vee \tau = \{[r]\}}}\prod_{B' \in \si}\mu\left(\{B \in \pi: B \subset B'\}, \{B'\}\right) = \binom{k}{\a}\sum_{\substack{\si \geq \pi,\\ \si \vee \tau = \{[r]\}}}\mu(\pi,\si),
\end{multline*}
where $\tau$ is the partition given by \cref{eq:tau}. Here, the last equality is a consequence of \cref{EqValueMobius} for the M\"obius function $\mu(\pi,\si)$.
Now, notice that partition $\tau$ is constructed in a way that $\tau
\geq \pi$, and the inequality is strict whenever $\#\supp(\a)>1$. Thus, we can apply \cref{prop:Weisner} to get
\[ \left[\prod_{B \in \pi}D^{\a(B)}u_B\right] \RHS = \begin{cases}\binom{k}{\a}\sum_{\substack{\si \geq \pi,\\ \si \vee \tau = \{[r]\}}}\mu(\pi,\si) &\text{ if } \#\supp(\a) = 1,\\ 0 &\text{ otherwise. }\end{cases}\]
But if $\#\supp(\a) = 1$ then $\binom{k}{\a} = 1$, and $\tau = \pi$,
thus $\si \vee \tau = \si$ (since $\si \geq \pi = \tau$). So condition $\si
\vee \tau = \{[r]\}$ implies that $\si = \{[r]\}$, which gives that
\[ \left[\prod_{B \in \pi}D^{\a(B)}u_B\right] \RHS = \begin{cases} \mu(\pi,\{[r]\}) &\text{ if } \#\supp(\a) = 1,\\ 0 &\text{ otherwise. }\end{cases}.\]
Comparing it with \cref{eq:LHS}, we can see that
\[ \left[\prod_{B \in \pi}D^{\a(B)}u_B\right] \LHS = \left[\prod_{B \in \pi}D^{\a(B)}u_B\right] \RHS,\]
which finishes the proof.
\end{proof}

\subsection{A multiplicative criterion for small cumulants}

% Let $R$ be a ring and $q$ a formal parameter.
% Denote $R(q)$ the field of rational functions in $q$ with coefficients in $R$.
% In all applications in this paper, $q$ is the deformation parameter of Macdonald polynomials.

% \begin{definition}
% \label{def:Osymbol}
% We use the following notation: for $r \in R$, and $f,g\in R(q)$,
% we write $f=O_r(g)$ if the rational function 
% $\frac{f(q)}{g(q)}$ has no pole in $q = r$.
% \end{definition}

Let $R$ be a ring and $q$ a formal parameter. We consider a family $\uu=(u_I)_{I\subseteq [\r]}$ 
of elements of $R(q)$ indexed by subsets of $[\r]$.
Throughout this section, we also assume that these elements are {\bf non-zero}
and $u_\emptyset=1$.

In addition to partial cumulants,
we also define the {\em cumulative factorization error terms} $T_H(\uu)$
of the family $\uu$.
The quantities $T_H(\uu)_{H \subseteq [\r],|H| \ge 2}$  are
inductively defined as follows:
for any subset $G$ of $[\r]$ of size at least $2$,
\begin{equation*}
u_G = \prod_{g \in G} u_{\{g\}} \cdot  
\prod_{H \subseteq G \atop |H| \ge 2} (1 + T_H(\uu)).
% \label{EqDefTInduction}
\end{equation*}
Using the inclusion-exclusion principle, 
a direct equivalent definition is the following:
for any subset $H$ of $[\r]$ of size at least 2, set
\begin{equation}
T_H(\uu) = \left(\prod_{G \subseteq H} u_G^{(-1)^{|H|-|G|}}\right) - 1.
\label{EqDefTDirect}
\end{equation}

F\'eray (using a different
framework) \cite{Feray2013} proved the following statement, which was
reproved in our recent joint paper with F\'eray \cite[Proposition 2.3]{DolegaFeray2016} using the framework of the current paper:
\begin{proposition}
    \label{PropEquivalenceSCQF}
    The following statements are equivalent:
    \begin{enumerate}[label=\Roman*]
        \item \hspace{-.2cm}. \label{ItemQuasiFactorization}
            {\em Strong factorization property when $q=r$:}
            for any subset $H \subseteq [\r]$ of size at least $2$, one has
            \begin{equation*}% \label{EqQuasiFactorization}
                T_{H}(\uu)
                = O_r\left((q-r)^{|H|-1}\right).
            \end{equation*}
        \item \hspace{-.2cm}. \label{ItemSmallCumulants}
            {\em Small cumulant property when $q=r$:}
            for any subset $H \subseteq [\r]$ of size at least $2$, one has
            \begin{equation*}% \label{EqSmallCumulants}
                \ka_H(\uu) = \left( \prod_{h \in H} u_h \right) O_r\left((q-r)^{|H|-1}\right). 
            \end{equation*}
    \end{enumerate}
\end{proposition}

\begin{remark}
In fact, above proposition was proved in the case $r=0$, but it is
enough to shift indeterminate $q \mapsto q-r$ to obtain the general result.
\end{remark}

A first consequence of this multiplicative criterion for small cumulants
is the following stability result.
\begin{corollary}
    Consider two families $(u_I)_{I \subseteq [\r]}$
    and $(v_I)_{I \subseteq [\r]}$ with the small cumulant property when $q \to r$.
    Then their entry-wise product $(u_I v_I)_{I \subseteq [\r]}$ and quotient
    $(u_I/v_I)_{I \subseteq [\r]}$ also have the small cumulant
    property when $q \to r$.
    \label{corol:stable_product}
\end{corollary}
\begin{proof}
    This is trivial for the strong factorization property and
    the small cumulant property is equivalent to it.
\end{proof}

Here is another consequence:

\begin{corollary}
\cref{ThmSFJack} is equivalent to \cref{conj:Macdonald}.
\end{corollary}

\begin{proof}
Let us fix a positive integer $r$, and partitions
$\la^1,\dots,\la^r$. For any subset $I \subset [r]$, define $u_I :=
J^{(q,t)}_{\la^I}$. Then, \cite[Chapter VI, Remark (8.4)-(iii)]{Macdonald1995} states that Macdonald polynomial $J^{(1,t)}_{\la^I}$ at $q = 1$ has non-zero
limit, thus $u_I, \left(u_I\right)^{-1} = O_1(1)$, and the statement
is an immediate consequence of \cref{PropEquivalenceSCQF}.
\end{proof}
\bigskip

\subsection{Hook cumulants}
We use the multiplicative criterion above to prove that families constructed from the hook polynomial defined by
\cref{eq:HookProduct} have the small cumulant properties at
$q=1$. This result is an important ingredient in the proof of the main result.

\begin{lemma}
Fix a positive integer $r$ and a subset $K$ of $[r]$.
Let $c \in \N$ and $(c_i)_{i \in K}$ be a family of some nonnegative integers, and let $C \neq 1 \in R$.
For a subset $I$ of $K$, we define
\[ v_I=1-C\cdot q^{c + \sum_{i \in I}c_i}\]
Then we have, for any subset $H$ of $K$,
\[ T_H(\vv) = O_1\left((q-1)^{|H|}\right).\]
\label{LemUsefulExample}
\end{lemma}

\begin{proof}
    It is enough to prove the statement for $H=K$.
    Indeed, the case of a general set $H$ follows by considering
    the same family restricted to subsets of $H$.
    
Define $R_\ev$ (resp. $R_\odd$) as
\[\prod_\delta \left( 1-C \cdot q^{c + \sum_{i \in \delta}c_i} \right),\]
where the product runs over subsets of $K$ of even (resp. odd)
size. Without loss of generality, we can assume that $|K|$ is even (the case when $|K|$ is odd is analogous).
With this notation,
$T_K(\vv)=R_\ev/R_\odd -1=(R_\ev-R_\odd)/R_\odd$.
Since $R_\odd^{-1}=O_1(1)$ (each term in the product is $O_1(1)$, as well as its inverse),
it is enough to show that $R_\ev-R_\odd=O_1\left((q-1)^{|K|}\right)$.

It is clear that 
\[ R_\ev\big|_{q=1} = R_\odd\big|_{q=1} = (1-C)^{2^{|K|}-1}.\]
Let us fix a positive integer $l < |K|$. 
Expanding the product in the definition of $R_\ev$ in the basis
$\{(q-1)^j\}_{j \geq 0}$, and using the binomial formula, one gets
\[\left[(q-1)^l\right]R_\ev = \sum_{1 \leq i \leq l}(1-C)^{2^{|K|-1}-i}C^i \ \frac{1}{i!} 
\sum_{\delta_1,\dots,\delta_i} \ \ \sum_{\substack{j_1+\cdots+ j_i = l \\ j_1,\dots,j_i \geq 1}} \prod_{1 \leq m \leq i}\binom{|\delta_m|_c}{j_m}.\]
The index set of the second summation symbol is the list of sets of $i$ distinct
(but not necessarily disjoint) subsets of $K$ of even size, and 
\[ |\delta|_c := c + \sum_{i \in \delta}c_i.\] The factor
$\frac{1}{i!}$ in the above formula comes from the fact that we should
sum over sets of $i$ distinct subsets of $K$, instead of lists, but it
is the same as the summation over the set of lists of $i$ distinct
subsets of $K$ and dividing by the number of permutations of $[i]$. 
Strictly from this formula it is clear that $[(q-1)^l]R_\ev$ is a symmetric polynomial in $c_i: i \in K$ of degree at most $l$. 
Of course, a similar formula with subsets of odd size holds for $[(q-1)^l]R_\odd$, which shows that it is a symmetric polynomial in $c_i: i \in K$ of degree at most $l$, as well.
For any positive integers $n,k$ we define a set $\Y(n,k)$ of sequences of $n$ nonnegative, nonincreasing integers, which are of the following form:
\[ \Y(n,k) = \{(\la,0^{n-\ell(\la)})\colon \la \in \Y_k,\ \ell(\la) \leq n\}.\]

It is well known (see for example \cite[Theorem 2.1]{KnopSahi1996}) that if $f,g$ are two symmetric polynomials of degree at most $k$ in $n$ indeterminates, then
\[ f = g \iff \forall \bm{x} \in \Y(n,k)\ \  f(\bm{x}) = g(\bm{x}). \]
Thus, in order to show that $[(q-1)^l]R_\ev = [(q-1)^l]R_\odd$ it is enough to show that this equality holds for all $(c_i)_{i \in K} \in \Y(|K|,l)$. Note that since $l < |K|$, then $c_k$ is necessarily equal to $0$, where $k$ is the biggest possible $k \in K$. It means that the function 
\[ f: (K)_{\ev}:=\{\delta \subset K: \delta \text{ has even size }\} \to (K)_{\odd}:=\{\delta \subset K: \delta \text{ has odd size }\}\] given by $f(\delta) := \delta \nabla \{k\}$, where $\nabla$ is the symmetric difference operator, is a bijection which preserves the following statistic $|\delta|_c = |f(\delta)|_c$.

Thus one has 
\begin{multline*}
\left[(q-1)^l\right]R_\ev = \sum_{1 \leq i \leq l}(1-C)^{2^{|K|-1}-i}C^i \ \frac{1}{i!} 
\sum_{\delta_1,\dots,\delta_i \in (K)_{\ev}} \ \ \sum_{\substack{j_1+\cdots+ j_i = l \\ j_1,\dots,j_i \geq 1}} \prod_{1 \leq m \leq i}\binom{|\delta_m|_c}{j_m}\\
= \sum_{1 \leq i \leq l}(1-C)^{2^{|K|-1}-i}C^i \ \frac{1}{i!} 
\sum_{\delta_1,\dots,\delta_i \in (K)_{\ev}} \ \ \sum_{\substack{j_1+\cdots+ j_i = l \\ j_1,\dots,j_i \geq 1}} \prod_{1 \leq m \leq i}\binom{|f(\delta_m|_c)}{j_m}\\
= \sum_{1 \leq i \leq l}(1-C)^{2^{|K|-1}-i}C^i \ \frac{1}{i!} 
\sum_{\delta_1,\dots,\delta_i \in (K)_{\odd}} \ \ \sum_{\substack{j_1+\cdots+ j_i = l \\ j_1,\dots,j_i \geq 1}} \prod_{1 \leq m \leq i}\binom{|\delta_m|_c}{j_m} = \left[(q-1)^l\right]R_\odd.
\end{multline*}
Since $l < |K|$ was an arbitrary positive integer, we have shown that
\[ R_\ev-R_\odd=O_1\left((q-1)^{|K|}\right), \]
which finishes the proof.
\end{proof}

\begin{proposition}
Fix some partitions $\la^1$, \dots, $\la^r$
and for a subset $I$ of $[r]$ set $u_I=h_{(q,t)}\left(\la^I\right)$.
    The family $(u_I)$ has the strong factorization, and hence, the
    small cumulant properties when $q \to 1$.
    \label{PropKumuHook}
\end{proposition}
\begin{proof}
    Fix some subset $I=\{i_1,\dots,i_t\}$ of $[r]$ with $i_1<\cdots<i_t$.
    Observe that the Young diagram $\la^I$ can be constructed 
    by sorting the columns of the diagrams $\la^{i_1}$, \ldots, $\la^{i_t}$ in decreasing order.
    When several columns have the same length, we put first the columns of $\la^{i_1}$,
    then those of $\la^{i_2}$ and so on; see \cref{fig:LegArmSum}
    (at the moment, please disregard symbols in boxes).
    This gives a way to identify boxes of $\la^I$ with boxes of the diagrams $\la^{i_s}$ ($1 \le s \le t$)
    that we shall use below.

    With this identification, if $b=(c,r)$ is a box in $\la^{g}$ for some $g \in I$, 
    its leg-length in $\la^I$ is the same as in $\la^{g}$.
    We denote it by $\ell(b)$.

    However, the arm-length of $b$ in $\la^I$ may be bigger 
    than the one in $\la^{g}$.
    We denote these two quantities by $a_I(b)$ and $a_{g}(b)$.
    Let us also define $a_i(b)$ for $i \neq g$ in $I$, as follows:
    \begin{itemize}
        \item for $i<g$, $a_i(b)$ is the number of boxes $b'$ in the $r$-th row of $\la^i$
            such that the size of the column of $b'$ is {\bf smaller than} the size of the column of $b$
            ({\em e.g.}, on \cref{fig:LegArmSum}, for $i=1$, these are
            boxes with a diamond);
        \item for $i>g$, $a_i(b)$ is the number of boxes $b'$ in the $r$-th row of $\la^i$
            such that the size of the column of $b'$ is {\bf at most} the size of the column of $b$
            ({\em e.g.}, on \cref{fig:LegArmSum}, for $i=3$, these are
                   boxes with an asterisk).
    \end{itemize}
    Looking at \cref{fig:LegArmSum}, it is easy to see that
    \begin{equation}
        a_I(b)= \sum_{i \in I} a_i(b).
        \label{EqArmLengthInSum}
    \end{equation}
    \begin{figure}[t]
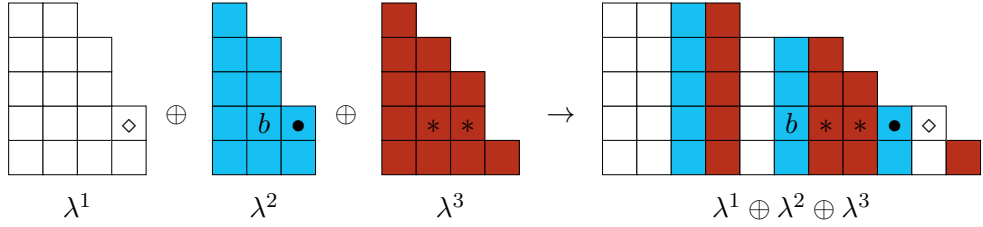

        \[
        \begin{array}{c}
        	    \YFrench
            \young(\ \ \ \ ,\ \ \ \diamond,\ \ \ ,\ \ \ ,\ \ ) \vspace{2mm} \\
        \la^1
    \end{array}\ \oplus\ \ 
    \begin{array}{c}
            \YFrench
            \Yfillcolour{cyan!70}
            \young(\ \ \ ,\ b\bullet,\ \ ,\ \ ,\ ) \vspace{2mm} \\
        \la^2
    \end{array}\ \oplus\ \ 
    \begin{array}{c}
            \YFrench
            \Yfillcolour{BrickRed}
            \young(\ \ \ \ ,\ \ast\ast,\ \ \ ,\ \ ,\ ) \vspace{2mm} \\
        \la^3
    \end{array}
    \ \to\
    \begin{array}{c}
        \newcommand\yb{\Yfillcolour{cyan!70}}
        \newcommand\yr{\Yfillcolour{BrickRed}}
        \newcommand\yw{\Yfillcolour{white}}
        \YFrench
        \young(\ \ !\yb\ !\yr\ !\yw\ !\yb\ !\yr\ \ !\yb\ !\yw\ !\yr\ ,!\yw\ \ !\yb\ !\yr\ !\yw\ !\yb b!\yr\ast\ast!\yb\bullet!\yw\diamond,!\yw\ \ !\yb\ !\yr\ !\yw\ !\yb\ !\yr\ \ ,!\yw\ \ !\yb\ !\yr\ !\yw\ !\yb\ !\yr\ ,!\yw\ \ !\yb\ !\yr\ )
        \vspace{2mm} \\
        \la^1\oplus\la^2\oplus\la^3
    \end{array}\vspace{-4mm}
        \]
        \caption{The diagram of an entry-wise sum of partitions.}
        \label{fig:LegArmSum}
    \end{figure}
    Therefore, for $G \subseteq [r]$, one has:
    \[u_G=h_{(q,t)}\left( \bigoplus_{g \in G} \la^g \right)
    = \prod_{g \in G} 
    \prod_{b \in \la^g} \left(1-q^{a_G(b)}t^{\ell(b)+1}\right). \] 
    From the definition of $T_{[r]}(\uu)$, given by \cref{EqDefTDirect}, we get:
    \begin{multline}
        1 + T_{[r]}(\uu) = \prod_{G \subseteq [r]} \left( \prod_{g \in G}                      
            \prod_{b \in \la^g} \left(1-q^{a_G(b)}t^{\ell(b)+1}\right) \right)^{(-1)^{r-|G|}} \\
            =\prod_{g \in [r]} \prod_{b \in \la^g}
            \left( \prod_{G \subseteq [r] \atop  G \ni g} \left(1-q^{a_G(b)}t^{\ell(b)+1}\right)^{(-1)^{r-|G|}}
            \right).
            \label{EqTechnicalTHook}
    \end{multline}
    The expression inside the bracket corresponds to $1+T_{[r] \setminus \{g\}}(\vv^b)$, where $\vv^b$ is defined
    as follows: if $I$ is a subset of $[r] \setminus \{g\}$, then
    \[v^b_I = \left(1-q^{a_{I\cup\{g\}}(b)}t^{\ell(b)+1}\right).\]
    Plugging \cref{EqArmLengthInSum} into definition of $v^b_I$, we observe that $v^b_I$ is as in \cref{LemUsefulExample}
    with the following values of the parameters: $K=[r] \setminus \{g\}$, 
    $C=t^{\ell(b)+1}$, $c = 1$, and $c_i = a_i(b)$ for $i \neq g$.
    Therefore, we conclude that
    \[T_{[r] \setminus \{g\}}(\vv^b)= O_1\left((q-1)^{r-1}\right).\]
    Going back to \cref{EqTechnicalTHook}, we have:
    \[1 + T_{[r]}(\uu) =\prod_{g \in [r]} \prod_{b \in \la^g} \left(1 + T_{[r] \setminus \{g\}}(\vv^b)\right)=
    1+O_1\left((q-1)^{r-1}\right),\]
    which completes the proof.
\end{proof}

We finish this section by presenting an important corollary from the above result.

\begin{proposition}
\label{cor:CumulantsTopMonomialExpansion}
For any partitions $\la^1, \dots, \la^r$ the cumulant $\kumu^\J(\la^1,\dots,\la^r)$ has a monomial expansion of the following form
\[ \kumu^\J(\la^1,\dots,\la^r) = \sum_{\mu \prec \la^{[r]}} c^{\la^1,\dots,\la^r}_\mu m_\mu + O_1\left((q-1)^{r-1}\right), \]
where 
\[c^{\la^1,\dots,\la^r}_\mu \in \begin{cases} \Z[q,t] &\text{ for } |\mu| = |\la^{[r]}|,\\ \Z[q,t^{-1},t] &\text{ for } |\mu| < |\la^{[r]}|.\end{cases}\]
%in the monomial basis as a polynomial in $\a$ when $\a \to 0$.
\end{proposition}

\begin{proof}
First, observe that for any partitions $\nu^1$ and $\nu^2$, one has
\[ m_{\nu^1}m_{\nu^2} = m_{\nu^1 \oplus \nu^2} + \sum_{\mu < \nu^1 \oplus \nu^2}b^{\nu^1,\nu^2}_\mu m_\mu,\]
for some integers $b^{\nu^1,\nu^2}_\mu$.

Fix partitions $\la^1, \dots, \la^r$ and a set partition $\pi=\{\pi_1,\cdots,\pi_s\} \in \PPP([r])$.
Note that $\la^{\pi_1} \oplus \cdots \oplus \la^{\pi_s}=\la^{[r]}$.
Thanks to \cref{eq:MacdonaldInMonomial} and the above observation on products of monomials,
there exist some coefficients 
\[d^{\la^{\pi_1},\cdots,\la^{\pi_s}}_\mu \in \begin{cases} \Z[q,t]
  &\text{ for } |\mu| = |\la^{[r]}|,\\ \Z[q,t^{-1},t] &\text{ for }
  |\mu| < |\la^{[r]}|\end{cases}\] 
such that:
\[\J_{\la^{\pi_1}}\cdots \J_{\la^{\pi_s}} = h_{(q,t)}(\la^{\pi_1})\cdots h_{(q,t)}(\la^{\pi_s}) m_{\la^{[r]}} + \sum_{\mu \prec \la^{[r]}}d^{\la^{\pi_1},\cdots,\la^{\pi_s}}_\mu m_\mu.\]
As a consequence, there exist some coefficients 
\[c^{\la^1,\dots,\la^r}_\mu \in \begin{cases} \Z[q,t] &\text{ for }
  |\mu| = |\la^{[r]}|,\\ \Z[q,t^{-1},t] &\text{ for } |\mu| <
  |\la^{[r]}|.\end{cases}\] 
such that
\[ \kumu^\J(\la^1,\dots,\la^r) = \kappa_{[\r]}(\vv)m_{\la^{[r]}} + \sum_{\mu < \la^{[r]}} c^{\la^1,\dots,\la^r}_\mu m_\mu,\]
where 
$v_I=h_{(q,t)}\left( \la^{I} \right)$.
\cref{PropKumuHook} completes the proof.
\end{proof}

%Let us now look at the second hook-polynomial $h'_\a$. 
%If we try to follow the same argument as above, we want to apply \cref{LemUsefulExample}
%with $K=[r] \setminus \{g\}$, $C=\ell(b)+\a (1+\, a_g(b))$, and $c_i = a_i(b)$ for $i \neq g$.
%Note, however, that if the box $b$ has leg-length $0$, then $C = 0$ for $\a = 0$,
%and in this case the hypothesis $C^{-1}=O(1)$ of \cref{LemUsefulExample} is not fulfilled.
%To overcome this difficulty, we define
%\[h''_\a(\la)= \prod_{\square \in \la \atop \ell(\square) \neq 0} \left(\a a(\square) + \ell(\square) +\a \right).\]
%By definition, the top-most box of each column of a diagram $\la$ has leg-length $0$.
%Moreover $\la$ has $m_i(\la^t)$ columns of height $i$, thus the
%arm-length of the
%top-most boxes of these columns are $0$, $1$,\ldots, $m_i(\la^t)-1$ respectively.
%Finally
%\[\prod_{\square \in \la \atop \ell(\square) = 0} \left(\a a(\square) + \ell(\square) +\a \right)
%= \a^{\la_1} \prod_i m_i(\la^t)!,\]
%so that
%\begin{equation}
%    h'_\a(\la)= \a^{\la_1} \, \left( \prod_i m_i(\la^t)! \right)\, h''_\a(\la).
%    \label{EqHPrimeHDPrime}
%\end{equation}
%Besides, the exact same proof than for $h_\a$ yields the following result:
%\begin{proposition}
%Fix some partitions $\la^1$, \dots, $\la^r$
%and, for a subset $I$ of $[r]$, set $v_I=h''_\alpha\left(\la^I \right)$.
%    The family $(v_I)$ has the strong factorization, and hence, the small cumulant properties.
%    \label{PropKumuHook2}
%\end{proposition}

\section{Differential operator and cumulant of interpolation Macdonald polynomials}
\label{sec:proof}

Let us fix partitions $\la^1,\dots,\la^r$, and for any subset $I
\subseteq [r]$ we define $u_I := \J^{(q^{-1},t^{-1})}_{\la^I}$. The purpose of
this section is an analysis of the action of the differential operator $D$
-- defined in \cref{eq:D} -- on the cumulant
$\kumu^J(\la^1,\dots,\la^r) = \kumu_{[r]}(\uu)$ with parameters
$q^{-1}$, and $t^{-1}$. In particular, this
analysis leads to the proofs of two crucial lemmas used in the proof
of \cref{theo:CumulantsMacdonald}. 

\subsection{Analysis of the decomposition}

For any positive integer $r$ and for any partitions $\la^1, \dots, \la^r$ we define
\begin{equation}
\label{eq:InclusionExclusion}
\IE_j(\la^1,\dots,\la^r) := \sum_{I \subseteq [r]}(-1)^{r-|I|}b^N_j\left(\la^I\right),
\end{equation}
where $b^N_j\left(\la^I\right)$ is given by \cref{eq:BinomialEigenvalue}.

\begin{proposition}
\label{prop:CumulantOfPartitionsVanishes}
Let $r > j \geq 1$ be positive integers. Then, for any partitions $\la^1, \dots, \la^r$ one has:
\[ \IE_j(\la^1,\dots, \la^r) = 0.\]
\end{proposition}

\begin{proof}
Expanding the definition and completing partitions with zeros, we have:
\[ \IE_j(\la^1,\dots,\la^r) = \sum_{1 \leq i \leq N} \sum_{I \subseteq [r]}(-1)^{r-|I|} \binom{\la^I_i}{j}t^{N-i}.\]
In particular, we have to prove that the summand corresponding to any given $1 \leq i \leq N$ is equal to $0$.
In other terms, we have to show that the polynomial
\[ \sum_{I \subseteq [r]}(-1)^{r-|I|} \binom{\bm{x}_I}{j} = 0,\]
where $\bm{x} = (x_1,\dots,x_r)$, and $\bm{x}_I := \sum_{i \in I}x_i$.

Note that it is a symmetric polynomial in $\bm{x}$ without constant term of degree at most $j$, thus it is enough to show that the coefficient of $\bm{x}^\mu := x_1^{\mu_1}\cdots x_r^{\mu_r}$ is equal to zero for all nonempty partitions $\mu$ of size at most $j$.
This coefficient is given by:
\[ \binom{|\mu|}{\mu_1,\dots,\mu_r}\frac{s(j,|\mu|)}{j!}\sum_{[\ell(\mu)] \subseteq I \subseteq [r]}(-1)^{r-|I|},\]
where $s(j,k)$ is the Stirling number of the first kind, i.e.
\[ (x)_j := x(x-1)\cdots(x-j+1) = \sum_{0 \leq k \leq j}s(j,k)x^{k}.\]
Since $\ell(\mu) \leq |\mu| \leq j < r$, we have that
\[ \sum_{[\ell(\mu)] \subseteq I \subseteq [r]}(-1)^{r-|I|} = 0,\]
which finishes the proof.
\end{proof}

We recall that
\[ \widetilde{D}\kumu_{[r]}(\uu) = \widetilde{D}\kumu^\J(\la^1,\dots,\la^r) = \sum_{\pi \in
  \PPP([r])}\mu\big(\pi,\{[r]\}\big)\widetilde{D}\left(\J^{(q^{-1},t^{-1})}_{\la^B} : B \in
  \pi\right),\]
where
\[ \widetilde{D}\left(\J^{(q^{-1},t^{-1})}_{\la^B} : B \in
  \pi\right) = \sum_{B \in \pi}\left(D \J^{(q^{-1},t^{-1})}_{\la^B}\ \cdot
  \prod_{B' \in \pi\setminus\{B\}}\J^{(q^{-1},t^{-1})}_{\la^{B'}}\right).\]

\begin{lemma}
\label{lem:A1SimpleForm}
For any positive integer $r \geq 2$ and any partitions $\la^1, \dots, \la^r$, the following equality holds true:
\begin{equation*}
\widetilde{D}\kumu^{\J}(\la^1,\dots,\la^r) = \sum_{j \geq 1}(q-1)^j\sum_{\substack{\si \in \PPP([r])\\ \#\si\leq j}} \IE_j\left(\la^B: B \in \si\right) \left(\prod_{B \in \si}\kumu^{\J}(\la^i: i
  \in B)\right).
\end{equation*}
\end{lemma}

\begin{proof}
Note that strictly from the definition of interpolation Macdonald
polynomials given by \cref{PropDefMacdonald} we know that for any set partition $\pi \in \PPP([r])$ the following identity holds:
\begin{multline*}
\sum_{B \in \pi}\left(D \J^{(q^{-1},t^{-1})}_{\la^B}\ \cdot \prod_{B' \in
  \pi\setminus\{B\}}\J^{(q^{-1},t^{-1})}_{\la^{B'}}\right) = \left(\sum_{B \in \pi}\ev\left({\la^B}\right)\right)\prod_{B \in
  \pi}\J^{(q^{-1},t^{-1})}_{\la^{B'}}\\
= \sum_{j \geq 1}(q-1)^j\left(\sum_{B \in \pi}b^N_j\left(\la^{\pi_i}\right)\right)\prod_{B \in
  \pi}\J^{(q^{-1},t^{-1})}_{\la^{B'}}.
\end{multline*}
If we substitute it into the definition of
$\widetilde{D}\kumu^{\J}(\la^1,\dots,\la^r)$, we have that
\begin{align*}
% \label{eq:D'1}
\widetilde{D}\kumu^{\J}(\la^1,\dots,\la^r) &= \sum_{j \geq 1}(q-1)^j  \sum_{\pi \in \PPP([r])} \left( \mu(\pi,\{[r]\}) \left(\sum_{B \in \pi}b^N_j(\la^B) \right)
\prod_{B \in \pi}\J^{(q^{-1},t^{-1})}_{\la^B} \right)\\
&=\sum_{j \geq 1}(q-1)^j  \sum_{\pi \in \PPP([r])} \left( \mu(\pi,\{[r]\}) \left(\sum_{B \in \pi}b^N_j(\la^B) \right)
\prod_{B \in \pi}u_B \right).
\end{align*}

Thanks to \cref{EqMomCum}, we can replace each occurence of $u_B$ in the above
equation by $\sum_{\substack{\pi \in \PPP(B) } } 
 \prod_{B' \in \pi} \ka_{B'}(\uu)$ to obtain the
following identity
\begin{multline*}
\widetilde{D}\kumu_{[r]}(\uu) = \sum_{j \geq 1}(q-1)^j\\
\cdot\sum_{\si \in \PPP([r])}\left(\sum_{\pi \in \PPP([r]); \pi \geq \si} \mu(\pi,\{[r]\}) \left(\sum_{B \in \pi}b^N_j(\la^B)\right)\right)\prod_{B \in \si}\kumu_B(\uu).
\end{multline*}
Fix a set partition $\si \in \PPP([r])$. We claim that the expression
in the bracket in the above equation is given by the following formula
\begin{equation}
\label{eq:takietam}
\sum_{\pi \in \PPP([r]); \pi \geq \si} \mu(\pi,\{[r]\}) \left(\sum_{B \in \pi}b^N_j(\la^B)\right) = \IE_j\left(\la^B: B \in \si\right),
\end{equation}
which finishes the proof, since the right hand side of \cref{eq:takietam} vanishes for all set partitions $\si$ such that $\# (\si) > j$, which is ensured by \cref{prop:CumulantOfPartitionsVanishes}.

Let us order the blocks of $\si$ in some way $\si = \{B_1,\dots,B_{\#\si}\}$.
The partitions $\pi$ coarser than $\si$ are in bijection with partitions of the blocks of $\si$,
that is partitions of $[\#\si]$.
Therefore the left-hand side of \cref{eq:takietam} can be rewritten as:
\begin{multline*} 
\sum_{\pi \in \PPP([r]); \pi \geq \si} \mu(\pi,\{[r]\}) \left(\sum_{B \in \pi}b^N_j(\la^B)\right) =\\
\sum_{\rho \in \PPP([\#\si])} \mu(\rho,\{[\#\si]\}) \left(\sum_{C \in \rho}b^N_j\left( \bigoplus_{j \in C} \la^{B_j} \right)\right).
\end{multline*}
Fix some subset $C$ of $[\#\si]$. 
The coefficient of $b^N_j\left( \bigoplus_{j \in C} \la^{B_j} \right)$
in the above sum is equal to
\[a_C:=\sum_{\rho \in \PPP([\#\si]) \atop C \in \rho} \mu(\rho,\{[\#\si]\}). \]
The set partitions $\rho$ of $[\#\si]$ that have $C$ as a block write uniquely as $C \cup \rho'$,
where $\rho'$ is a set partition of $[\# (\si)]\setminus C)$. Thus
\begin{multline*} 
a_C = \sum_{\rho' \in \PPP([\# (\si)]\setminus C)} \mu(C \cup \rho', \{[\# (\si)]\}) \\
= \sum_{0 \leq i \leq \# (\si)-|C|} S(\#\si-|C|,i) i! (-1)^i = (-1)^{\#\si-|C|}, 
\end{multline*}
where $S(n,k)$ is the \emph{Stirling number of the second kind} 
and the last equality comes from the relation
\[ \sum_{0 \leq k \leq n}S(n,k)(x)_k = x^n\]
evaluated at $x=-1$ (here,
$(x)_k := x(x - 1)\cdots(x-k+1)$ denotes the falling factorial).
This finishes the proof of \cref{eq:takietam}, and also completes the proof of the lemma.
\end{proof}

\begin{lemma}
\label{lem:A2SimpleForm}
For any positive integer $r$ and any partitions $\la^1, \dots, \la^r$, the following equality holds true
\begin{multline}
\label{eq:A2Identity}
\widetilde{D}\ \kumu^{\J}(\la^1,\dots,\la^r) =
\widetilde{D}\ \kumu_{[r]}(\uu) \\
= \sum_j \frac{(q-1)^j}{j!}\sum_{1 \leq i \leq N}A_i(\bm{x};t)(x_i^j-x_i^{j-1})\sum_{\substack{\pi \in \PPP([r]), \\ \#\pi \leq j} }\sum_{\substack{\a \in \N_+^\pi, \\ |\a| = j} } \binom{j}{\a} \prod_{B \in \pi} \left(D_i^{\a(B)}\kumu_{B}(\uu)\right).
\end{multline}
\end{lemma}

\begin{proof}
We recall that
\[D = \sum_{j \geq 1}\frac{(q-1)^{j}}{j!}\sum_iA_i(\bm{x};t)(x_i^j-x_i^{j-1}) D^j_i\]
(which is given by \cref{eq:DAsDifferential}),
where $D^j_i := \frac{\partial^j}{\partial x_i^j}$. Since
$\frac{\partial}{\partial x_i} \in \Der_{\Symm_N}$ is a derivation, we
have the following formula
\[ \widetilde{D}\ \kumu_{[r]}(\uu) = \sum_{j \geq 1}\frac{(q-1)^{j}}{j!}\sum_i
A_i(\bm{x};t)(x_i^j-x_i^{j-1}) \left(\widetilde{D^j_i}\ \kumu_{[r]}(\uu) \right),
\]
and one can apply \cref{lem:zasada} and substitute the following identity
\[ \widetilde{D^j_i}\ \kumu_{[r]}(\uu) = \sum_{\substack{\pi \in \PPP([r]), \\ \#\pi \leq j} }\sum_{\substack{\a \in \N_+^\pi, \\ |\a| = j} } \binom{j}{\a} \prod_{B \in \pi} \left(D_i^{\a(B)}\kumu_{B}(\uu)\right),\] 
which immediately gives \cref{eq:A2Identity}.

\end{proof}

\section*{Acknowledgments}

I thank Valentin F\'eray for encouraging discussions, and Piotr Śniady
for many valuable remarks.

\bibliographystyle{alpha}

\bibliography{biblio2015}

\end{document}